\documentclass{article}

\usepackage{amsthm}
\usepackage{amsmath, amssymb}
\usepackage{hyperref}
\usepackage{tikz}
\usetikzlibrary{matrix}
\usepackage[matrix, arrow, curve]{xy}
\usepackage{xcolor}

\newtheorem{Thm}{Theorem}[section]
\newtheorem{Def}[Thm]{Definition}

\newtheorem{Prop}[Thm]{Proposition}
\newtheorem{Kor}[Thm]{Corollary}
\newtheorem{Rem}[Thm]{Remark}

\title{Some remarks on Spin-orbits of unit vectors}

\author{Tariq Syed\\
Institut f{\"u}r Mathematik\\
Johannes Gutenberg-Universit{\"a}t Mainz\\
Staudingerweg 9\\
55128 Mainz, Germany\\
tariq.syed@gmx.de}

\date{\today
}
 
\begin{document}

\maketitle

\begin{abstract}
For $n \in \mathbb{N}$ and a commutative ring $R$ with $2 \in R^{\times}$, the group $SL_n (R)$ acts on the set $Um_n (R)$ of unimodular vectors of length $n$ and $Spin_{2n}(R)$ acts on the set of unit vectors $U_{2n-1}(R)$. We give an example of a ring for which the comparison map $Um_n (R)/SL_n (R) \rightarrow U_{2n-1}(R)/Spin_{2n}(R)$ fails to be bijective.
\end{abstract}

\tableofcontents

\section{Introduction}

Let $R$ be a commutative ring with $2 \in R^{\times}$ and $n \in \mathbb{N}$. We denote by $Um_n (R)$ the set of unimodular vectors of length $n$, i.e., the set of vectors $a={(a_{1},...,a_{n})}^{t} \in R^n$ with $a_{i} \in R$, $1 \leq i \leq n$, such that $\langle a_{1},...,a_{n} \rangle = R$; furthermore, we denote by $U_{2n-1}(R)$ the corresponding set of unit vectors, i.e., the set of elements $(a,b) \in R^n \oplus R^n$ with $\sum_{i=1}^{n} a_i b_i = 1$, where $a={(a_{1},...,a_{n})}^{t}$ and $b={(b_{1},...,b_{n})}^{t}$. The group $SL_n (R)$ and hence any of its subgroups act on $Um_n (R)$, while the group $Spin_{2n}(R)$ acts on $U_{2n-1}(R)$. The orbit spaces $Um_n (R)/SL_n (R)$ play a central role in the study of stably free modules over commutative rings (e.g., see \cite{FRS}, \cite{S} or \cite{Sy2}). There are well-defined comparison maps between orbit spaces
\begin{center}
$Um_{n}(R)/SL_n (R) \rightarrow U_{2n-1}(R)/Spin_{2n}(R)$,
\end{center}
which are automatically surjective. Vineeth Chintala has proven in \cite{C2} that the analogous comparison maps
\begin{center}
$Um_{n}(R)/E_n (R) \rightarrow U_{2n-1}(R)/Epin_{2n}(R)$
\end{center}
between the corresponding elementary orbit spaces are bijective for $n \geq 3$; this laid the foundation for the study of stably free $R$-modules via Spin-orbits of unit vectors. It is therefore natural to ask whether the comparison maps between $SL_n (R)$-orbits and $Spin_{2n}(R)$-orbits are also bijective for commutative rings when $n \geq 3$.\\
While we observe that the comparison map is bijective for some specific classes of rings (cf. Theorem \ref{bijectivity}), we give a negative answer to this question in this paper. For any field $k$ with $char(k) \neq 2$ and $n \in \mathbb{N}$, we let $S_{2n-1} = k[x_{1},...,x_{n},y_{1},...,y_{n}]/\langle \sum_{i=1}^{n} x_{i}y_{i} - 1 \rangle$. While the comparison map for unimodular vectors of length $3$ is bijective for $S_5$ (cf. Corollary \ref{S5}), we prove the following result for $S_7$:

\paragraph{Theorem.} If $R=S_7$ and $k$ is infinite perfect with $char(k) \neq 2$, then the comparison map $Um_3 (R)/SL_3 (R) \rightarrow U_{5}(R)/Spin_6 (R)$ is not injective.\\\\
In particular, the ring $R=S_7$ gives an example of a ring for which the comparison map between $SL_{n}(R)$-orbits and $Spin_{2n}(R)$-orbits fails to be bijective.

%\paragraph{Theorem.} The map $Um_{3}(R)/SL_{3}(R) \rightarrow U_{5}(R)/Spin_{6}(R)$ is a bijection if
%\begin{itemize}
%\item $R$ is a Noetherian ring of dimension $\leq 2$
%\item $R$ is a smooth affine algebra of dimension $3$ over a perfect field $k$ such that $6 \in k^{\times}$ and $c.d.(k) \leq 1$
%\item $R$ is a smooth affine algebra of dimension $4$ over an algebraically closed field $k$ with $6 \in k^{\times}$.
%\end{itemize}

The paper is structured as follows: In Section \ref{2.1} we introduce the comparison maps between the orbit space of unimodular vectors under the action of $SL_{n}(R)$ and the orbit space of unit vectors under the action of $Spin_{2n}(R)$; in Section \ref{2.2} we introduce the degree maps and discuss the connection between the degree maps and Spin-orbits of unit vectors. Section \ref{2.3} serves as a brief introduction to $\mathbb{A}^1$-homotopy theory as needed for this paper. Finally, in Section \ref{3} we will prove some bijectivity results for the comparison maps above and eventually prove the main result of this paper.

\section*{Acknowledgements}
The author would like to thank the anonymous referee for suggesting changes which greatly improved the exposition of the paper. The author would like to thank Alexey Ananyevskiy, Vineeth Chintala, Jean Fasel and Ravi Rao for helpful comments on this work. The author was funded by the Deutsche Forschungsgemeinschaft (DFG, German Research Foundation) - Project number 461453992.

\section{Preliminaries}

Throughout this paper, $R$ will always denote a commutative ring with $2 \in R^{\times}$.

\subsection{The comparison maps}\label{2.1}

At first, we introduce the set of unimodular vectors:

\begin{Def}
For $n \in \mathbb{N}$, we denote be $Um_{n} (R)$ the set of unimodular vectors of length $n$ over $R$, i.e., the set of vectors $a={(a_{1},...,a_{n})}^{t} \in R^n$ with $a_{i} \in R$, $1 \leq i \leq n$, such that $\langle a_{1},...,a_{n} \rangle = R$.
\end{Def}

The group $GL_{n} (R)$ of invertible matrices of rank $n$ over $R$ acts on the left on $Um_{n}(R)$ by matrix multiplication. In particular, all subgroups of $GL_{n}(R)$ and hence $SL_{n}(R)$ and $E_{n}(R)$ act on $Um_{n}(R)$. In this paper unimodular vectors are usually thought of as column vectors, but in the literature unimodular vectors are often defined as unimodular rows or unimodular row vectors; in the latter case, $GL_n (R)$ and its subgroups obviously act on the right on the set of unimodular rows of length $n$ over $R$ by matrix multiplication. The two definitions make no essential difference: As a matter of fact, transposition of vectors and matrices induces a bijection between the orbit space of unimodular column vectors of length $n$ over $R$ under the left action of $GL_n (R)$ and the orbit space of unimodular row vectors of length $n$ over $R$ under the right action of $GL_n (R)$; analogous statements hold for the subgroups $E_n (R)$ and $SL_n (R)$ of $GL_n (R)$. Therefore we will just speak of unimodular vectors of length $n$ over $R$ and usually think of them as unimodular column vectors and elements of $R^n$ in this paper. Now consider the space $H(R^n) = R^n \oplus R^n$ equipped with the quadratic form
\begin{center}
$q(a,b) = \sum_{i=1}^{n} a_i b_i$, 
\end{center}
where $a = {(a_{1},...,a_{n})}^{t}, b={(b_{1},...,b_{n})}^{t} \in R^n$.

\begin{Def}
The set $U_{2n-1} (R)$ of unit vectors is the set of elements $(a,b)$ in $H(R^n)$ with $q(a,b) = 1$.
\end{Def}

Recall now from \cite[Chapter IV, \S 1]{K} that the Clifford algebra $Cl(V,q)$ of $V=H(R^n)$ is the quotient of the tensor algebra
\begin{center}
$T(V) = R \oplus V \oplus V^{\otimes 2} \oplus ...$
\end{center}

by the two-sided ideal generated by the elements of the form $x \otimes x - q(x)$ for $x \in V$. Grading $T(V)$ by even and odd degrees, $Cl(V,q)$ inherits a $\mathbb{Z}/2\mathbb{Z}$-grading. Furthermore, there is an inclusion $V \hookrightarrow Cl(V,q), (v,w) \mapsto (v,w)$.\\
Now let $a={(a_{1},...,a_{n})}^{t}, b={(b_{1},...,b_{n})}^{t} \in R^n$. Following the notation in \cite{AF2} (and not the notation of \cite{C1} or \cite{C2}), the Suslin matrix can be defined inductively by $\alpha_{n}(a,b) = (a_{1})$ if $n=1$ and

\begin{center}
$\alpha_{n} (a,b) = \begin{pmatrix}
a_{1} {id}_{2^{n-2}} & \alpha_{n-1} (a',b') \\
-{\alpha_{n-1} (b',a')}^{t} & b_{1} {id}_{2^{n-2}}
\end{pmatrix}$
\end{center}

if $n \geq 2$, where $a'={(a_{2},...,a_{n})}^{t}, b'={(b_{2},...,b_{n})}^{t} \in R^{n-1}$. In \cite[Lemma 5.1]{S} Suslin proved that $\det (\alpha_{n} (a,b)) = {(a^{t} b)}^{2^{n-2}}$ if $n \geq 2$; in particular, if $a = {(a_{1},...,a_{n})}^{t}$ is a unimodular vector of length $n$ and $b={(b_{1},...,b_{n})}^{t}$ defines a section of $a$, i.e., $q (a,b) =\sum_{i=1}^{n} a_{i} b_{i} = 1$, then $\alpha_{n} (a,b) \in SL_{2^{n-1}} (R)$.\\
Similarly, one defines $\overline{\alpha_{n}(a,b)} = (b_{1})$ if $n=1$ and

\begin{center}
$\overline{\alpha_{n} (a,b)} = \begin{pmatrix}
b_{1} {id}_{2^{n-2}} & -\alpha_{n-1} (a',b') \\
{\alpha_{n-1} (b',a')}^{t} & a_{1} {id}_{2^{n-2}}
\end{pmatrix}$
\end{center}

if $n \geq 2$, where again $a'={(a_{2},...,a_{n})}^{t}, b'={(b_{2},...,b_{n})}^{t} \in R^{n-1}$. It was proven in \cite[Lemma 5.1]{S} that one has ${\alpha_{n}(b,a)}^{t} = \overline{\alpha_{n}(a,b)}$ and $\alpha_{n}(a,b)\overline{\alpha_{n}(a,b)} = \overline{\alpha_{n}(a,b)}\alpha_{n}(a,b) = (a \cdot b^t)id_{2^{n-1}}$.\\
Altogether, one obtains maps
\begin{center}
$\alpha_{n}: U_{2n-1}(R) \rightarrow GL_{2^{n-1}}(R)$
\end{center}
and
\begin{center}
$\overline{\alpha_{n}}: U_{2n-1}(R) \rightarrow GL_{2^{n-1}}(R)$.
\end{center}
One can use Suslin matrices in order to obtain a description of the Clifford algebra $Cl(V,q)$ (cf. \cite[Section 2]{C1} and \cite[Section 2.3]{C1}): Let us denote by $M_{2^n}(R)$ the set of $2^n \times 2^n$-matrices over $R$. The assignment

\begin{center}
$\phi (a,b) = \begin{pmatrix}
0 & \alpha_{n}(a,b) \\
\overline{\alpha_{n}(a,b)} & 0
\end{pmatrix}$
\end{center}

descends to a map $Cl(V,q) \rightarrow M_{2^n}(R)$ and is an isomorphism of $\mathbb{Z}/2\mathbb{Z}$-graded algebras, where matrices of the form

\begin{center}
$\begin{pmatrix}
g_{1} & 0 \\
0 & g_{2}
\end{pmatrix}$
\end{center}

are of degree $0$ and matrices of the form

\begin{center}
$\begin{pmatrix}
0 & g_{1} \\
g_{2} & 0
\end{pmatrix}$
\end{center}

are of degree $1$ (cf. \cite[Theorem 2.3 and remarks thereafter]{C1}). Furthermore, one can inductively define $2^{n-1} \times 2^{n-1}$-matrices $J_{n}$ by $J_{1} = 1$ if $n=1$ and

\begin{center}
$J_{n} = \begin{pmatrix}
0 & J_{n-1} \\
-J_{n-1} & 0
\end{pmatrix}$
\end{center}

if $n \geq 2$ is even or

\begin{center}
$J_{n} = \begin{pmatrix}
J_{n-1} & 0 \\
0 & -J_{n-1}
\end{pmatrix}$
\end{center}

if $n \geq 2$ is odd. Note again that we use the notation from \cite{AF2}. By induction, one sees that $J_n J_n^t = J_n^t J_n = id_{2^{n-1}}$ and $J_n^{-1} = {(-1)}^{\frac{n(n-1)}{2}} J_n$. It follows that $M^{\ast} = J_{n+1} M^t J^t_{n+1}$, where $M \in M_{2^n}(R)$, defines an involution on $M_{2^n}(R)$. As one can check, this involution corresponds to the so-called standard involution of the Clifford algebra $Cl(V,q)$ (cf. \cite[Theorem 3.1]{C1}).\\
Now let us denote by $Cl = Cl(V,q) = Cl_0 \oplus Cl_1$ the $\mathbb{Z}/2\mathbb{Z}$-grading of $Cl(V,q)$.

\begin{Def}
For any $n \in \mathbb{N}$, the Spin group is defined as $Spin_{2n}(R):=\{g \in Cl_0 |~g g^\ast  = 1, g H(R^n) g^{-1} = H(R^n)\}$.
\end{Def}
%$x x^\ast$ already implies $x \in Cl_0^{\ast}$: Just check it for matrices of degree 0 via the isomorphism $Cl \cong M_{2^n}(R)$

Following \cite[Chapter IV, \S 6]{K}, there is a map

\begin{center}
$\pi: Spin_{2n}(R) \rightarrow SO_{2n} (R)$
\end{center}

defined by $\pi(g): R^{2n} \rightarrow R^{2n}, (a,b) \mapsto g \cdot (a,b) \cdot g^{-1}$. If we let $U_{2n-1}(R) :=\{(a,b) \in R^n \oplus R^n | q(a,b) = 1\}$ as above, then it is clear that $O_{2n}(R)$ and hence $SO_{2n}(R)$ act on the left on $U_{2n-1}(R)$ by matrix multiplication. Consequently, one obtains a left action of $Spin_{2n}(R)$ on $U_{2n-1}(R)$ via $\pi$. If $n \geq 3$, one has a map

\begin{center}
$Um_{n}(R) \rightarrow U_{2n-1}(R)/SO_{2n}(R)$
\end{center}

which sends any $v \in Um_{n}(R)$ to the orbit of $(v,w)$, where $w$ is any section of $v$. It follows for example from \cite[Theorem 4.3]{C2} that this map is well-defined for $n \geq 3$.\\
Thus, from now on, let us assume $n \geq 3$. Let $v_{1},v_{2} \in Um_{n}(R)$ with sections $w_{1}$ and $w_{2}$ and assume $\sigma v_1 = v_2$ for some $\sigma \in SL_{n}(R)$. Then

\begin{center}
$H(\sigma) = \begin{pmatrix}
\sigma & 0 \\
0 & {\sigma^{t}}^{-1}
\end{pmatrix} \in SO_{2n}(R)$
\end{center}

and hence $(v_{1},w_{1}) \sim_{SO_{2n}(R)} H(\sigma) (v_{1},w_{1}) = (\sigma v_{1},{\sigma^{t}}^{-1}w_{1}) \sim_{SO_{2n}(R)} (v_{2},w_{2})$, where the latter equivalence holds again because of \cite[Theorem 4.3]{C2}. In particular, it follows that the map $Um_{n}(R) \rightarrow U_{2n-1}(R)/SO_{2n}(R)$ descends to a map
\begin{center}
$Um_{n}(R)/SL_{n}(R) \rightarrow U_{2n-1}(R)/SO_{2n}(R)$.
\end{center}
Furthermore, there is an injective homomorphism $\wedge: SL_{n} (R) \rightarrow Spin_{2n}(R)$ which lifts the homomorphism $H:SL_{n}(R) \rightarrow SO_{2n}(R)$ along $\pi$ (cf. \cite[Chapter IV, \S 6]{K}), so by using again \cite[Theorem 4.3]{C2} one obtains a sequence of surjections between orbit spaces
\begin{center}
$Um_{n}(R)/SL_{n}(R) \rightarrow U_{2n-1}(R)/Spin_{2n}(R) \rightarrow U_{2n-1}(R)/SO_{2n}(R)$.
\end{center}

%\begin{Rem}
%If $n \geq 3$, it follows again from \cite[Theorem 4.3]{C2} and the fact that $Epin_{2n}(R)$ is a normal subgroup of $Spin_{2n}(R)$ that we can equivalently consider the orbit space $Um_{n}(R)/Spin_{2n}(R)$ instead of $U_{2n-1}(R)/Spin_{2n}(R)$, i.e. the action of $Spin_{2n}(R)$ on $U_{2n-1}(R)$ descends to an action $Um_{n}(R)$ and the orbit spaces coincide.
%\end{Rem}

We now try to compare $SL_{n}(R)$-orbits of $Um_{n}(R)$ with $Spin_{2n}(R)$-orbits of $U_{2n-1}(R)$. So let us investigate the injectivity of the comparison map
\begin{center}
$Um_{n}(R)/SL_{n}(R) \rightarrow U_{2n-1}(R)/Spin_{2n}(R)$
\end{center}
above. First of all, let us denote by $St (x,y)$ the subgroup of $Spin_{2n}(R)$ which stabilizes $(x,y) \in U_{2n-1}(R)$. We now prove the following criterion for the injectivity of the comparison map:

\begin{Thm}\label{Spin-criterion}
The natural map $Um_{n}(R)/SL_{n}(R) \rightarrow U_{2n-1}(R)/Spin_{2n}(R)$ above is injective (and hence bijective) if and only if, for any $(x,y) \in U_{2n-1}(R)$, one has $Spin_{2n}(R) = \wedge (SL_{n}(R)) Epin_{2n}(R)St (x,y)$. Furthermore, the orbit space $Um_{n}(R)/SL_{n}(R)$ is trivial if and only if the orbit space $U_{2n-1}(R)/Spin_{2n}(R)$ is trivial and $Spin_{2n}(R) = \wedge(SL_{n}(R))Epin_{2n}(R)St (u_{n})$, where $u_{n} = (e_{n},e_{n})$ and $e_{n}= {(1,0,...,0)}^{t} \in R^n$.
\end{Thm}

\begin{Rem}\label{RemarkPermutingGroups}
Let $G$ be a group. For two subgroups $M,N \subset G$, let $MN = \{mn | m \in M, n \in N\}$ and $NM = \{nm|n \in N, m \in M\}$. Recall that $M$ and $N$ are called permuting if $MN = NM$; this is the case if and only if the set $MN$ is a subgroup of $G$. Now recall that $Epin_{2n}(R)$ is a normal subgroup of $Spin_{2n}(R)$ for $n \geq 3$: Indeed, $EO_{2n}(R)$ is a normal subgroup of $SO_{2n}(R)$ (cf. \cite[Theorem 2.12]{SK}) and $Epin_{2n}(R)$ is the preimage of $EO_{2n}(R)$ under the homomorphism $\pi$ above (cf. \cite[Remarks after Proposition 4.3.3]{B}). It follows that $Epin_{2n}(R)$ permutes with any subgroup of $Spin_{2n}(R)$ and, moreover, $Spin_{2n}(R)/Epin_{2n}(R)$ is a group. We denote by $\overline{\wedge SL_{n}(R)}$ and $\overline{St(x,y)}$ the images of $\wedge (SL_{n}(R))$ and $St(x,y)$ under the canonical projection $Spin_{2n}(R) \rightarrow Spin_{2n}(R)/Epin_{2n}(R)$ respectively. The criterion in Theorem \ref{Spin-criterion} above essentially means that $Spin_{2n}(R)/Epin_{2n}(R)= \overline{\wedge SL_{n}(R)}~\overline{St(x,y)}$ for any $(x,y) \in U_{2n-1}(R)$. In particular, it follows that the criterion in the theorem is equivalent to the equality $Spin_{2n}(R) = St(x,y) Epin_{2n}(R) \wedge (SL_{n}(R))$ for any $(x,y) \in U_{2n-1}(R)$. Furthermore, it follows completely analogously that the equality $SO_{2n}(R) = H(SL_{n}(R)) EO_{2n}(R)St (x,y)$ in Theorem \ref{SO-criterion} below is equivalent to the equality $SO_{2n}(R) = St(x,y) EO_{2n}(R) H(SL_{n}(R))$.
\end{Rem}

\begin{proof}
Let $\varphi \in Spin_{2n}(R)$ such that $(v,w) =  \varphi \cdot (x,y)$. Note that elements of $St (x,y)$ do not change the right-hand vector $(x,y)$ by definition. Elements of $Epin_{2n}(R)$ will not change the $E_{n}(R)$-orbits of the unimodular vector $x$ in question because of \cite[Theorem 4.4]{C2}. Hence if one has $Spin_{2n}(R) =  \wedge (SL_{n}(R)) Epin_{2n}(R) St (x,y)$, then $v$ will be in the same $SL_{n}(R)$-orbit as $x$.\\
Conversely, let us fix $(x,y) \in U_{2n-1}(R)$. Assume that $(v,w) \sim_{Spin_{2n}(R)} (x,y)$ always implies $v \sim_{SL_{n}(R)} x$ and let $\varphi \in Spin_{2n}(R)$. Let $(v,w) =  \varphi \cdot (x,y)$. Then there is $\varphi' \in SL_{n}(R)$ such that $\wedge(\varphi') \cdot (v,w)= (\varphi' v, {{\varphi'}^{t}}^{-1} w) = (x, y')$ for some $y'$. Furthermore, there is $\varphi'' \in Epin_{2n}(R)$ with $\varphi'' \cdot (x,y')=(x,y)$ by \cite[Theorem 4.3]{C2}. But then $\varphi'' \wedge(\varphi') \varphi \in St (x,y)$. So $\varphi \in \wedge (SL_{n}(R)) Epin_{2n}(R)St (x,y)$ and, for any $(x,y) \in U_{2n-1}(R)$, $Spin_{2n}(R) = \wedge (SL_{n}(R)) Epin_{2n}(R)St (x,y)$. This proves the first statement in the theorem.\\
The second statement in the theorem follows easily from the first paragraph of this proof by realizing that one only has to consider the vector $u_n$ in order to show that $Um_{n}(R)/SL_{n}(R)$ is trivial if $U_{2n-1}(R)/Spin_{2n}(R)$ is trivial.
\end{proof}

Note that $St (u_{n})$ is just $Spin_{2n-1}(R)$. Of course, the criterion in the theorem is automatically satisfied whenever the orbit space $Um_{n}(R)/SL_{n}(R)$ is trivial. Later we will be able to give more examples of rings for which the criterion in Theorem \ref{Spin-criterion} holds.\\
Analogously, we now investigate the injectivity of the comparison map
\begin{center}
$Um_{n}(R)/SL_{n}(R) \rightarrow U_{2n-1}(R)/SO_{2n}(R)$.
\end{center}
By abuse of notation, let us also denote by $St (x,y)$ the subgroup of $SO_{2n}(R)$ which stabilizes $(x,y) \in U_{2n-1}(R)$; note that $St (u_{n})$ is just $SO_{2n-1}(R)$. We prove the following criterion for the equality of $SL_{n}(R)$-orbits of unimodular vectors and $SO_{2n}(R)$-orbits of unit vectors:

\begin{Thm}\label{SO-criterion}
The natural map $Um_{n}(R)/SL_{n}(R) \rightarrow U_{2n-1}(R)/SO_{2n}(R)$ is injective (and hence bijective) if and only if, for any $(x,y) \in U_{2n-1}(R)$, one has $SO_{2n}(R) = H(SL_{n}(R)) EO_{2n}(R)St (x,y)$. Furthermore, $Um_{n}(R)/SL_{n}(R)$ is trivial if and only if the orbit space $U_{2n-1}(R)/SO_{2n}(R)$ is trivial and $SO_{2n}(R) = H(SL_{n}(R)) EO_{2n}(R) St (u_{n})$, where $u_{n} = (e_{n},e_{n})$ and $e_{n}= {(1,0,...,0)}^{t} \in R^n$.
\end{Thm}

\begin{proof}
The proof is very similar to the proof of Theorem \ref{Spin-criterion} above, but for the reader's convenience we still give a full proof: Let $\varphi \in SO_{2n}(R)$ such that $(v,w) = \varphi \cdot (x,y)$. Again, note that elements of $St (x,y)$ do not change the right-hand vector $(x,y)$ by definition. Furthermore, elements of $EO_{2n}(R)$ will not change the $E_{n}(R)$-orbits of the corresponding unimodular vector $x$ because of \cite[Theorem 4.4]{C2} again. Thus, if $SO_{2n}(R) = H(SL_{n}(R)) EO_{2n}(R)St (x,y)$, then $v$ lies in the same $SL_{n}(R)$-orbit as $x$.\\
Conversely, we fix $(x,y) \in U_{2n-1}(R)$. Assume that $(v,w) \sim_{SO_{2n}(R)} (x,y)$ always implies $v \sim_{SL_{n}(R)} x$ and let $\varphi \in SO_{2n}(R)$. Let $(v,w) = \varphi \cdot (x,y)$. By assumption, there is $\varphi' \in H(SL_{n}(R))$ such that $H(\varphi') \cdot (v,w) = (\varphi' v, {{\varphi'}^{t}}^{-1} w) = (x, y')$ for some $y'$. By \cite[Theorem 4.3]{C2}, there is $\varphi'' \in EO_{2n}(R)$ with $\varphi'' \cdot (x,y') =(x,y)$. But then $\varphi''H(\varphi')\varphi \in St (x,y)$. So $\varphi \in H(SL_{n}(R)) EO_{2n}(R)St (x,y)$ and, for any $(x,y) \in U_{2n-1}(R)$, $SO_{2n}(R) = H(SL_{n}(R)) EO_{2n}(R) St (x,y)$. This proves the first statement in the theorem.\\
Again, the second statement in the theorem follows from the first paragraph by realizing that one only has to consider the vector $u_n$ in order to show that $Um_{n}(R)/SL_{n}(R)$ is trivial in case $U_{2n-1}(R)/SO_{2n}(R)$ is trivial.
\end{proof}

\subsection{Degree maps}\label{2.2}

In this section, we introduce the degree maps already studied in \cite{S} and \cite[Section 3]{AF2} and discuss their relationship with Spin-orbits of unit vectors.\\
For all $n \in \mathbb{N}$, we denote by $Sp_{2n}(R)$ the group of symplectic matrices of rank $2n$, by $O_{2n}(R)$ the group of orthogonal matrices of rank $2n$, by $S_{2n}(R)$ the set of invertible symmetric matrices of rank $2n$, by $A_{2n}(R)$ the set of invertible alternating matrices of rank $2n$ and by $\tilde{A}_{2n}(R)$ its subset of invertible alternating matrices of rank $2n$ with Pfaffian $1$. The group $GL_{2n}(R)$ acts on the set $S_{2n}(R)$ by $M \mapsto \varphi^{t} M \varphi$ for $\varphi \in GL_{2n}(R)$ and $M \in S_{2n}(R)$. Similarly, $GL_{2n}(R)$ acts on the set $A_{2n}(R)$ by $M \mapsto \varphi^{t} M \varphi$ for $\varphi \in GL_{2n}(R)$ and $M \in A_{2n}(R)$; this action also induces an action of $SL_{2n}(R)$ on $\tilde{A}_{2n}(R)$. We denote by $S_{2n}(R)/GL_{2n}(R)$, $A_{2n}(R)/GL_{2n}(R)$ and $\tilde{A}_{2n}(R)/SL_{2n}(R)$ the corresponding orbit spaces.\\
For $m,n \geq 1$, there are embeddings $S_{2n}(R) \rightarrow S_{2n+2m}(R), M \mapsto M \perp \sigma_{2m}$, and $A_{2n}(R) \rightarrow A_{2n+2m}(R), M \mapsto M \perp \psi_{2m}$, where the matrices $\sigma_{2m}$ and $\psi_{2m}$ are defined inductively by
\begin{center}
$\sigma_2 =
\begin{pmatrix}
0 & 1 \\
 1 & 0
\end{pmatrix},~
\psi_2 =
\begin{pmatrix}
0 & 1 \\
- 1 & 0
\end{pmatrix}
$
\end{center}
and $\sigma_{2m} = \sigma_{2m-2} \perp \sigma_{2}$ and $\psi_{2m} = \psi_{2m-2} \perp \psi_{2}$ for $m \geq 2$. We denote by $S(R) = colim_n S_{2n}(R)$ and $A(R) = colim_n A_{2n}(R)$ the direct limits under the respective embeddings.\\
Now we discuss an equivalence relation on the set $S(R)$: Two matrices $M \in S_{2m}(R)$ and $N \in S_{2n}(R)$ are said to be equivalent, $M \sim N$, if there exists $i \geq 1$ and $E \in E_{2n+2m+2i}(R)$ with
\begin{center}
$M \perp \sigma_{2n+2i} = E^t (N \perp \sigma_{2m+2i}) E$.
\end{center}
It follows from \cite[Lemme 4.5.1.9]{BL} that the block sum of matrices equips the set of equivalence classes $S(R)/{\sim}$ with the structure of an abelian group. We also refer the reader to \cite[Section 2]{AF2} for a discussion of the group $S(R)/{\sim}$.\\
Analogously, we have an equivalence relation on the set $A(R)$: Two matrices $M \in A_{2m}(R)$ and $N \in A_{2n}(R)$ are said to be equivalent, $M \sim N$, if there exists $i \geq 1$ and $E \in E_{2n+2m+2i}(R)$ with
\begin{center}
$M \perp \psi_{2n+2i} = E^t (N \perp \psi_{2m+2i}) E$.
\end{center}
Then it follows from \cite[\S 3]{SV} that the block sum of matrices equips the set of equivalence classes $W'_E (R) := A(R)/{\sim}$ with the structure of an abelian group. One can also consider the direct limit $\tilde{A}(R) = colim_n \tilde{A}_{2n}(R)$ with respect to the induced embeddings $\tilde{A}_{2n}(R) \rightarrow \tilde{A}_{2n+2m}(R), M \mapsto M \perp \psi_{2m}$; then it follows again from \cite[\S 3]{SV} that the set of equivalence classes $W_E (R) := A (R)/{\sim}$ is an abelian group. This group is also called the elementary symplectic Witt group and clearly is a subgroup of $W'_E (R)$; we refer the reader to \cite[Section 3.A]{Sy1} for a discussion of the groups $W'_E (R)$ and $W_E (R)$.\\
One can also define a quotient of $W_E (R)$ which usually denoted $W_{SL}(R)$ as follows: In this case two matrices $M \in \tilde{A}_{2m}(R)$ and $N \in \tilde{A}_{2n}(R)$ are said to be equivalent, $M \sim_{SL} N$, if there exists $i \geq 1$ and $\varphi \in SL_{2n+2m+2i}(R)$ with
\begin{center}
$M \perp \psi_{2n+2i} = \varphi^t (N \perp \psi_{2m+2i}) \varphi$.
\end{center}
It follows again from \cite[\S 3]{SV} that the block sum of matrices equips the set of equivalence classes $W_{SL} (R) := \tilde{A}(R)/{\sim_{SL}}$ with the structure of an abelian group. We also refer the reader to \cite[Section 2.B]{Sy2} for a detailed discussion of this group.\\
If $R$ is a smooth affine algebra over a field $k$ with $char(k) \neq 2$, then it is argued in \cite[Section 2.3]{AF2} that the abelian group $S(R)/{\sim}$ can be identified with the higher Grothendieck-Witt group $GW_{1}^{1}(R)$; furthermore, the group $W'_E (R)$ can be identified with the higher Grothendieck-Witt group $GW_{1}^{3}(R)$. We refer the reader to \cite[Section 2]{AF2} for a detailed discussion of the higher Grothendieck-Witt groups $GW_{i}^{j}(R)$ for $i,j \in \mathbb{Z}$. For the purpose of this paper, we simply mention that there exist exact Karoubi periodicity sequences
\begin{center}
$K_{1}(R) \xrightarrow{H_{1,1}} GW_{1}^{1}(R) \xrightarrow{\eta} GW_{0}^{0}(R) \xrightarrow{f_{0,0}} K_{0}(R)$
\end{center}
and
\begin{center}
$K_{1}(R) \xrightarrow{H_{1,3}} GW_{1}^{3}(R) \xrightarrow{\eta} GW_{0}^{2}(R) \xrightarrow{f_{0,2}} K_{0}(R)$.
\end{center}
The groups $K_i (R)$, $i=0,1$, are the usual algebraic $K$-theory groups; the group $GW_{0}^{0} (R) = GW(R)$ is the Grothendieck-Witt group of non-degenerate symmetric bilinear forms over $R$, while the group $GW_{0}^{2}(R) = KSp_{0}(R)$ is the $0$th symplectic $K$-theory group (of non-degenerate alternating bilinear forms). The homomorphisms $f_{0,0}$ and $f_{2,0}$ are the usual forgetful homomorphisms, while the homomorphisms $H_{1,1}$ and $H_{1,3}$ are called hyperbolic homomorphisms. The hyperbolic homomorphism $H_{1,1}$ is defined by $M \mapsto M^t \sigma_{2n} M \in S_{2n}(R)$ for $M \in GL_{2n}(R)$, while $H_{1,3}$ is defined by $M \mapsto M^t \psi_{2n} M \in A_{2n}(R)$ for $M \in GL_{2n}(R)$. Following \cite[Section 2.4]{AF2}, the homomorphism $\eta$ in each sequence is defined by $M \mapsto [M] - [\sigma_{2n}] \in GW (R)$ for $M \in S_{2n}(R)$ and respectively by $M \mapsto [M] - [\psi_{2n}] \in KSp_{0}(R)$ for $M \in A_{2n}(R)$; as the definitions of the forgetful homomorphisms and the homomorphisms $H_{1,1}$, $H_{1,3}$ and $\eta$ make sense over any commutative ring $R$ with $2 \in R^{\times}$, we will use the notation for these homomorphisms in this general case as well. Finally, the group $W_{SL}(R)$ can be identifed with the cokernel of $H_{1,3}$ and the homomorphism $\eta$ identifies $W_{SL}(R)$ with the kernel of the forgetful homomorphism $KSp_{0} (R) \rightarrow K_{0}(R)$.\\
Now let $R$ be any commutative ring with $2 \in R^{\times}$ again. The $2n \times 2n$-matrices $\tau_{2n}$ over $R$ are defined inductively by
\begin{center}
$\tau_2 =
\begin{pmatrix}
0 & 1 \\
0 & 0
\end{pmatrix}$
\end{center}
and $\tau_{2n} = \tau_{2n-2} \perp \tau_{2}$ for $n \geq 2$. One also defines matrices $E_n \in GL_{2^{n-1}}(R)$ inductively by $E_n = id_{2^{n-1}}$ for $n =1,2$ and by
\begin{center}
\begin{equation*}
E_{n}:=
\begin{cases}
\begin{pmatrix}
1 & 0 \\
0 & J_{n-1}^{t}
\end{pmatrix}
\begin{pmatrix}
1 & 0 \\
\tau_{2^{n-2}} & 1
\end{pmatrix}
\begin{pmatrix}
1 & -\sigma_{2^{n-2}} \\
0 & 1
\end{pmatrix}
\begin{pmatrix}
1 & 0 \\
0 & \psi_{2^{n-2}}
\end{pmatrix}  & \text{if } n \equiv 0 \text{ mod } 4 \\
\begin{pmatrix}
E_{n-1} & 0 \\
0 & E_{n-1}
\end{pmatrix}
\begin{pmatrix}
1 & 0 \\
0 & \psi_{2^{n-2}}
\end{pmatrix}  & \text{if } n \equiv 1 \text{ mod } 4 \\
\begin{pmatrix}
1 & 0 \\
0 & J_{n-1}^{t}
\end{pmatrix}
\begin{pmatrix}
1 & 0 \\
\tau_{2^{n-2}} & 1
\end{pmatrix}
\begin{pmatrix}
1 & \psi_{2^{n-2}} \\
0 & 1
\end{pmatrix}
\begin{pmatrix}
1 & 0 \\
0 & \sigma_{2^{n-2}}
\end{pmatrix}  & \text{if } n \equiv 2 \text{ mod } 4 \\
\begin{pmatrix}
E_{n-1} & 0 \\
0 & E_{n-1}
\end{pmatrix}
\begin{pmatrix}
1 & 0 \\
0 & \sigma_{2^{n-2}}
\end{pmatrix}  & \text{if } n \equiv 3 \text{ mod } 4
   \end{cases}
\end{equation*}
\end{center}
for $n \geq 3$. For any $(v,w) \in U_{2n-1}(R)$, it follows from \cite[Lemma 5.3]{S} and \cite[Lemma 3.3.1]{AF2} that the matrices $E^{-1}_n \alpha_{n}^{t}(v,w) E_{n}$ are orthogonal if $n \equiv 0~mod~4$ and symplectic if $n \equiv 2~mod~4$ and that the matrices $E^{t}_n \alpha_{n}(v,w) J_{n} E_{n}$ are invertible symmetric if $n \equiv 1~mod~4$ and invertible alternating if $n \equiv 3~mod~4$.\\
Hence defining
\begin{center}
\begin{equation*}
\Psi_{n}(v,w):=
\begin{cases}
E^{-1}_n \alpha_{n}^{t}(v,w) E_{n} & \text{if n is even}\\
E^{t}_n \alpha_{n}(v,w) J_{n} E_{n} & \text{if n is odd}
\end{cases}
\end{equation*}
\end{center}
for any $(v,w) \in U_{2n-1}(R)$ defines a map
\begin{center}
\begin{equation*}
   \Psi_{n}: U_{2n-1}(R) \rightarrow
   \begin{cases}
     O_{2^{n-1}}(R) & \text{if } n \equiv 0 \text{ mod } 4 \\
     S_{2^{n-1}}(R) & \text{if } n \equiv 1 \text{ mod } 4 \\
     Sp_{2^{n-1}}(R) & \text{if } n \equiv 2 \text{ mod } 4 \\
     A_{2^{n-1}}(R) & \text{if } n \equiv 3 \text{ mod } 4
   \end{cases}
\end{equation*}
\end{center}
which we will call the degree map (see also \cite[Section 3.3]{AF2}).\\
Finally, let us now return to the study of Spin-orbits of unit vectors: If $n \geq 3$ is odd, the action of $Spin_{2n}(R)$ on $U_{2n-1}(R)$ can be described as follows (cf. \cite[Section 4]{C1}): Any $\varphi \in Spin_{2n}(R)$ has the form $\varphi = diag (g,{g^{\ast}}^{-1})$ for $g \in GL_{2^{n-1}}(R)$ under the identification $Cl (V,q) \cong M_{2^n}(R)$ (cf. \cite[Theorem 4.3]{C1}); here $g^{\ast}$ denotes the standard involution of $g \in M_{2^{n-1}}(R)$. If $(v,w) \in U_{2n-1} (R)$, then the matrix $g \alpha_{n}(v,w) g^{\ast}$ is again a Suslin matrix of some unit vectors, i.e., $g \alpha_{n}(v,w) g^{\ast} = \alpha_{n}(v',w')$ for some $(v',w') \in U_{2n-1}(R)$, and this is the action of $\varphi$ on $(v,w)$ (cf. \cite[Section 4.1]{C1}).\\
Moreover, if we associate to any Suslin matrix $\alpha_{n}(v,w)$ the matrix

\begin{center}
$\Psi_{n}(v,w) = E_{n}^{t} \alpha_{n}(v,w)J_{n}E_{n}$
\end{center}

from the definition of the degree map, this gives a bijection from Suslin matrices to matrices of the form $\Psi_{n}(v,w)$; in this sense, no information is lost from the Suslin matrix construction. Indeed, the map $\Psi_n$ is obtained from the Suslin matrix map $\alpha_{n}: U_{2n-1}(R) \rightarrow GL_{2^{n-1}}(R)$ by considering the composite of $\alpha_{n}$ and the bijection $GL_{2^{n-1}}(R) \rightarrow GL_{2^{n-1}}(R), M \mapsto E_n^t M J_n E_n$ and realizing that the image of this composite lies in fact in $S_{2^{n-1}}(R)$ if $n \equiv 1~mod~4$ or in $A_{2^{n-1}}(R)$ if $n \equiv 3~mod~4$. Furthermore, note that $\Psi_{n} (e_{n},e_{n}) = \sigma_{2^{n-1}}$ if $n \equiv 1~mod~4$ and that $\Psi_{n} (e_{n},e_{n}) = \psi_{2^{n-1}}$ if $n \equiv 3~mod~4$ (cf. \cite[Lemma 3.3.1]{AF2}).\\
The action of $Spin_{2n}(R)$ translates to

\begin{center}
$\Psi_{n}(v',w') = E_{n}^{t} g \alpha_{n}(v,w)g^{\ast} J_{n} E_{n}=E_{n}^{t} g \alpha_{n}(v,w)J_{n} g^{t}J_{n}^{t} J_{n} E_{n}=E_{n}^{t} g \alpha_{n}(v,w)J_{n} g^{t} E_{n}=(E_{n}^{t} g {E_{n}^{t}}^{-1})(E_{n}^{t} \alpha_{n}(v,w)J_{n}E_{n})(E_{n}^{-1} g^{t} E_{n})=g' \Psi_{n}(v,w) {g'}^{t}$,
\end{center}

where $(v',w')$ is again the action of $\varphi = diag(g,{g^{\ast}}^{-1})$ on $(v,w)$ and $g' = E_{n}^{t} g {E_{n}^{t}}^{-1}$. This description has some consequences that we now explain:\\
Let us first explain why the determinant of $\Psi_n (v,w)$ is $1$ if $n \equiv 1~mod~4$ and that the Pfaffian of $\Psi_n (v,w)$ is $1$ if $n \equiv 3~mod~4$; note that it suffices to check this locally, i.e., after localization at any prime ideal. But over a local ring every given unimodular vector $(v,w)$ is completable to a matrix in $E_n (R)$ and the homomorphism $\wedge: SL_n (R) \rightarrow Spin_{2n}(R)$ always maps $E_n (R)$ into $Epin_{2n}(R)$, so we can assume that some $\varphi=(g,g^{\ast}) \in Epin_{2n}(R)$ transforms $(v,w)$ into $u_n = (e_n,e_n)$; one can then check easily via the explicit generators of $Epin_{2n}(R)$ given at the end of \cite[Section 3.3]{C2} that the determinant of $g$ has to be $1$.\\
In particular, using the computation of $\Psi_{n}(v',w')$ above for $(v',w')=u_n$, it follows locally and then also over any commutative ring $R$ that the determinant of $\Psi_n (v,w)$ is indeed $1$ if $n \equiv 1~mod~4$ and that the Pfaffian of $\Psi_n (v,w)$ is $1$ if $n \equiv 3~mod~4$, as claimed. For any commutative ring $R$ and $\varphi = diag(g,{g^{\ast}}^{-1}) \in Spin_{2n}(R)$, it then follows again from the computation above that the determinant of $g'$ and hence of $g$ is a square root of $1$ if $n \equiv 1~mod~4$ and equal to $1$ if $n \equiv 3~mod~4$.\\
By composition we obtain induced maps
\begin{center}
$Um_{n}(R)/SL_{n}(R) \rightarrow U_{2n-1}(R)/Spin_{2n}(R) \rightarrow S_{2^{n-1}}(R)/GL_{2^{n-1}}(R) \rightarrow coker (K_{1}(R) \xrightarrow{H_{1,1}} S(R)/{\sim})$
\end{center}
for $n \equiv 1~mod~4$ and
\begin{center}
$Um_{n}(R)/SL_{n}(R) \rightarrow U_{2n-1}(R)/Spin_{2n}(R) \rightarrow \tilde{A}_{2^{n-1}}(R)/SL_{2^{n-1}}(R) \rightarrow coker (K_1(R) \xrightarrow{H_{1,3}} W'_E (R)) = W_{SL}(R)$
\end{center}
for $n \equiv 3~mod~4$. By abuse of notation, we call these composites again $\Psi_{n}$ or $\Psi_{n}$ modulo SL.\\
It also follows from the discussion above that the injective homomorphism $St (x,y) \rightarrow GL_{2^{n-1}}(R) \xrightarrow{\cong} GL_{2^{n-1}}(R), (g,{g^{\ast}}^{-1}) \mapsto g \mapsto g'=E_{n}^{t}g{E_{n}^{t}}^{-1}$ factors through $O(\Psi_{n}(x,y))$ if $n \equiv 1~mod~4$ and $Sp(\Psi_{n}(x,y))$ if $n \equiv 3~mod~4$.
\\
It is a natural question whether the map $\Psi_{n}$ is injective or, in other words, detects the non-triviality of $Um_{n}(R)/SL_{n}(R)$, i.e., $(v,w)$ and $(v',w')$ represent isomorphic stably free modules if and only if their images under $\Psi_{n}$ coincide.

\begin{Prop}\label{detect}
Let $R$ be a ring and $n \geq 3$ odd. If the degree map $\Psi_{n}$ modulo SL is injective,  then the natural map $Um_{n}(R)/SL_{n}(R) \rightarrow U_{2n-1}(R)/Spin_{2n}(R)$ above is bijective.
\end{Prop}

\begin{proof}
Follows directly from the factorizations above.
\end{proof}

\begin{Rem}\label{Vaserstein}
Recall from \cite{Sy1} and \cite{Sy2} that there is a generalized Vaserstein symbol $V_{\theta} : Um_3 (R)/SL_{3} (R) \rightarrow W_{SL}(R)$ modulo SL associated to any isomorphism $\theta: R \xrightarrow{\cong} \det (R^2)$ via the canonical identification $\tilde{\mathit{V}}_{SL} (\mathit{R}) \cong \mathit{W}_{\mathit{SL}} (\mathit{R})$ from \cite[Section 2.C]{Sy2}. If we let $\mathit{e}_{1} = {(1,0)}^{t}, \mathit{e}_{2} = {(0,1)}^{t} \in \mathit{R}^{2}$, then there is a canonical isomorphism $\theta: \mathit{R} \xrightarrow{\cong} \det(\mathit{R}^{2})$ given by $1 \mapsto \mathit{e}_{1} \wedge \mathit{e}_{2}$. The map $\Psi_{3}$ modulo SL coincides with the Vaserstein symbol modulo SL associated to $\theta$. However, the map $\Psi_{3}$ modulo SL does not coincide with the original Vaserstein symbol considered in \cite[\S 5]{SV}; as a matter of fact, the original Vaserstein symbol coincides with the generalized Vaserstein symbol associated to $-\theta$ via the canonical identification $\tilde{\mathit{V}}_{SL} (\mathit{R}) \cong \mathit{W}_{\mathit{SL}} (\mathit{R})$ (cf. \cite[Section 4.B]{Sy1}). Altogether, it follows from Proposition \ref{detect} that the comparison map $Um_{3}(R)/SL_{3}(R) \rightarrow U_{5}(R)/Spin_{6}(R)$ is bijective whenever the generalized Vaserstein symbol modulo SL associated to the isomorphism $\theta$ above is injective.\\
Furthermore, the assignment $(g,{g^{\ast}}^{-1}) \mapsto g$ induces  isomorphisms $Spin_6 (R) \cong SL_4 (R)$ and $Epin_6 (R) \cong E_4 (R)$ (cf. \cite[Theorems 6.1 and 7.5]{C1}); in particular, it follows that the assignment $(g,{g^{\ast}}^{-1}) \mapsto E_{3}^{t} g {E_{3}^{t}}^{-1}$ also induces an isomorphism $Spin_6 (R) \cong SL_4 (R)$. Furthermore, one can check easily that any matrix in $\tilde{A}_{4}(R)$ is of the form $\Psi_{3}(v,w)$ for some $(v,w) \in U_{5}(R)$ and therefore the assignment $(v,w) \mapsto \Psi_3 (v,w)$ induces a bijection $U_5 (R) \cong \tilde{A}_4 (R)$. Hence it follows that the map $U_{5}(R)/Spin_{6}(R) \rightarrow \tilde{A}_{4}(R)/SL_{4}(R)$ is bijective and the discussion above shows that $St (x,y)$ corresponds to $Sp (\chi)$ with $\chi = \Psi_{3} (x,y)$ under the isomorphism $Spin_{6}(R) \rightarrow SL_{4}(R), (g,{g^{\ast}}^{-1}) \mapsto E_{3}^{t} g {E_{3}^{t}}^{-1}$.
\end{Rem}

%The matrix $E_3$ normalizes the group $E_{4}(R)$, just check it on generators given in \cite{SV}.

If $n$ is even, the action of $Spin_{2n}(R)$ on $U_{2n-1}(R)$ can be described as follows (cf. \cite[Section 5]{C1}): Any $\varphi \in Spin_{2n}(R)$ has the form $\varphi = diag (\varphi_{1},\varphi_{2})$ for $\varphi_{1},\varphi_{2} \in GL_{2^{n-1}}(R)$ under the identification $Cl (V,q) \cong M_{2^n}(R)$. Furthermore, both $\varphi_{1}$ and $\varphi_{2}$ satisfy $\varphi_{i}^{\ast} = \varphi_{i}^{-1}$ for $i=1,2$; it follows from this that $E_{n}^{-1} \varphi_{i} E_{n} \in O_{2^{n-1}}(R)$ if $n \equiv 0~mod~4$ and $E_{n}^{-1} \varphi_{i} E_{n} \in Sp_{2^{n-1}}(R)$ if $n \equiv 2~mod~4$, because the equality $\varphi_{i}^{\ast} = J_{n} \varphi_{i}^{t} J_{n}^{t}$ implies
\begin{center}
$J_{n}^{t} = \varphi_{i}^t J_{n}^{t} \varphi_{i}$.
\end{center}
Using $E_{n}^t J_n^t E_n = \sigma_{2^{n-1}}$ if $n \equiv 0~mod~4$ and $E_{n}^t J_n^t E_n = -\psi_{2^{n-1}}$ if $n \equiv 2~mod~4$ (cf. \cite[Lemma 3.3.1]{AF2}), one obtains the statement.\\
If $(v,w) \in U_{2n-1} (R)$, then the matrix $\varphi_{1} \alpha_{n}(v,w) \varphi_{2}^{-1}$ is again a Suslin matrix of some unit vectors, i.e., $\varphi_{1} \alpha_{n}(v,w) \varphi_{2}^{-1} = \alpha_{n}(v',w')$ for some $(v',w') \in U_{2n-1}(R)$, and this is essentially the action of $\varphi$ on $(v,w)$.\\
Moreover, if we associate to any Suslin matrix $\alpha_{n}(v,w)$ the matrix

\begin{center}
$\Psi_{n}(v,w) = E_{n}^{-1} \alpha_{n}^{t}(v,w)E_{n}$,
\end{center}

from the definition of the degree map, this gives a bijection from Suslin matrices to matrices of the form $\Psi_{n}(v,w)$. The action of $Spin_{2n}(R)$ translates to
\begin{center}
$\Psi_{n}(v',w') = E_{n}^{-1} \varphi_{2}^{-t} \alpha_{n}^{t}(v,w) \varphi_{1}^{t} E_{n} = (E_{n}^{-1} \varphi_{2}^{-t} E_{n}) (E_{n}^{-1} \alpha_{n}^{t}(v,w) E_{n}) (E_{n}^{-1} \varphi_{1}^{t} E_{n}) = (E_{n}^{-1} \varphi_{2}^{-t} E_{n}) \Psi_{n}(v,w) (E_{n}^{-1} \varphi_{1}^{t} E_{n})$.
\end{center}

It follows from the equality $\varphi_{1} \alpha_{n}(v,w) \varphi_{2}^{-1} = \alpha_{n}(v',w')$ that always $det(\varphi_{1})=det(\varphi_{2})$ and $det (\varphi)$ is a square; for $(v,w) = (e_{n}, e_{n})$, it follows that $\varphi_{1}\varphi_{2}^{-1}$ is a Suslin matrix of some unit vectors, which we denote by $\alpha (\varphi)$.\\ %Assume we map the matrices $\Psi_{n}(v,w)$ and $\Psi_{n}(v',w')$ into an abelian group (e.g. abelianizations) and we want to compute their difference in this group; then an easy computation shows that this difference is just given by the image of $\Psi_{n}(\varphi):= E_{n}^{-1} {\alpha(\varphi)}^{t} E_n$.\\
Altogether, we have an embedding $Spin_{2n}(R) \rightarrow O_{2^{n-1}}(R) \times O_{2^{n-1}}(R)$ if $n \equiv 0~mod~4$ and $Spin_{2n}(R) \rightarrow Sp_{2^{n-1}}(R) \times Sp_{2^{n-1}}(R)$ if $n \equiv 2~mod~4$. Furthermore, if $\varphi \in St(e_{n},e_{n})$ it follows that $\varphi_{1}=\varphi_{2}$ and hence we have an embedding $Spin_{2n-1}(R) \rightarrow O_{2^{n-1}}(R)$ if $n \equiv 0~mod~4$ and $Spin_{2n-1}(R) \rightarrow Sp_{2^{n-1}}(R)$ if $n \equiv 2~mod~4$. Unfortunately, if $n$ is even, it is not clear at all whether there exist analogues of the $\Psi_n$ modulo SL maps as in the odd case.
%\\\\
%\textbf{Comment.} It could be reasonable to replace the usual definition $\Psi_{n}(v,w) = E_{n}^{-1} \alpha_{n}^{t}(v,w)E_{n}$ by $\Psi_{n}(v,w) = E_{n}^{-1} \alpha_{n}(v,w)E_{n}$ in order get nicer formulas for the action of $Spin_{2n}(R)$.

\subsection{Motivic homotopy theory}\label{2.3}

In this section we outline the construction of the unstable $\mathbb{A}^1$-homotopy category over a field given in \cite{MV}; for an introduction to model categories, we refer the reader to \cite{Ho}. The underlying idea of $\mathbb{A}^{1}$-homotopy theory is to develop a homotopy theory of schemes in which the affine line $\mathbb{A}^1$ plays the role of the unit interval $[0,1]$ in topology. So let $k$ be a fixed base field.\\
We denote by $Sm_k$ the category of smooth separated schemes of finite type over $k$ and then we consider the category $Spc_{k}=\Delta^{op}Shv_{Nis}(Sm_{k})$ of simplicial Nisnevich sheaves over $Sm_k$; we will also refer to simplicial Nisnevich sheaves over $Sm_k$ as spaces. Both the category $Sm_k$ and the category $\Delta^{op} Sets$ of simplicial sets can be embedded into $Spc_{k}$. Following \cite{MV}, there is a model structure on this category in which cofibrations are simply given by monomorphisms and weak equivalences are given by morphisms which induce weak equivalences of simplicials sets on stalks; this model structure is usually called the simplicial model structure and weak equivalences with respect to this model structure are also called simplicial weak equivalences. The $\mathbb{A}^1$-model structure is then obtained as a left Bousfield localization of the simplicial model structure with respect to the projection morphisms $\mathcal{X} \times \mathbb{A}^{1} \rightarrow \mathcal{X}$; its weak equivalences are called $\mathbb{A}^{1}$-weak equivalences. The associated homotopy category (which is obtained from $Spc_{k}$ by inverting $\mathbb{A}^{1}$-weak equivalences) is usually denoted $\mathcal{H} (k)$ and called the unstable $\mathbb{A}^{1}$-homotopy category over $k$. There is a pointed version of this category which is constructed completely analogously by considering the category $Spc_{k, \bullet}$ of pointed simplicial Nisnevich sheaves over $Sm_k$ (which are usually referred to as pointed spaces); it is denoted $\mathcal{H}_{\bullet} (k)$ and called the pointed unstable $\mathbb{A}^1$-homotopy category. The category $Spc_{k, \bullet}$ is a pointed model category and hence features the formalism of fiber and cofiber sequences (cf. \cite[Chapter 6]{Ho}).\\
As in topology, there is a smash product $(\mathcal{X},x) \wedge (\mathcal{Y},y)$ for two pointed spaces $(\mathcal{X},x)$ and $(\mathcal{Y},y)$. The simplicial suspension functor is defined as the functor $\Sigma_{s} = (S^1,\ast) \wedge -: Spc_{k, \bullet} \rightarrow Spc_{k, \bullet}$; it has a right adjoint $\Omega_{s}: Spc_{k, \bullet} \rightarrow Spc_{k, \bullet}$ called the loop space functor. This adjoint pair of functors forms a Quillen adjunction.\\
For two spaces $\mathcal{X},\mathcal{Y}$, we denote by $[\mathcal{X},\mathcal{Y}]_{\mathcal{H}(k)}$ the set of morphisms from $\mathcal{X}$ to $\mathcal{Y}$ in the category $\mathcal{H}(k)$; similarly, for two pointed spaces $(\mathcal{X},x),(\mathcal{Y},y)$, we denote by $[(\mathcal{X},x),(\mathcal{Y},y)]_{\mathcal{H}_{\bullet}(k)}$ the set of morphisms from $(\mathcal{X},x)$ to $(\mathcal{Y},y)$ in the pointed category $\mathcal{H}_{\bullet}(k)$.\\
Now let $n \geq 1$. Then we define $S_{2n-1} = k[x_{1},...,x_{n},y_{1},...,y_{n}]/\langle \sum_{i=1}^{n} x_{i}y_{i} - 1 \rangle$ and let $Q_{2n-1} = Spec(S_{2n-1})$ be the smooth affine hypersurface in $\mathbb{A}^{2n}$. The morphism $p_{2n-1}:Q_{2n-1} \rightarrow \mathbb{A}^{n}\setminus 0$ induced by projection on the coefficients $x_{1},...,x_{n}$ is locally trivial with fibers isomorphic to $\mathbb{A}^{n-1}$ and therefore an $\mathbb{A}^{1}$-weak equivalence (cf. \cite{AF1}, \cite{AF2}). In particular, we conclude that there is a pointed $\mathbb{A}^{1}$-weak equivalence

\begin{center}
$\mathbb{A}^{n}\setminus 0 \simeq_{\mathbb{A}^{1}} Q_{2n-1}$
\end{center}

if we let $\mathbb{A}^{n}\setminus 0$ have the basepoint $(1,0,..,0)$ and if we let $Q_{2n-1}$ have the basepoint $(1,0,..,0,1,0,..,0)$. Furthermore, if $R$ is an affine $k$-algebra and $X = Spec(R)$, it is easy to see that

\begin{center}
$\mathit{Um}_{n} (R) \cong Hom_{Sch_{k}} (X, \mathbb{A}^{n}\setminus 0)$
\end{center}

and

\begin{center}
$U_{2n-1}(R)=\{(a,b)|a,b \in R^{n}, q(a,b) = 1\} = Hom_{Sch_{k}} (X, Q_{2n-1})$,
\end{center}

where $Sch_k$ is the category of Noetherian $k$-schemes of finite Krull dimension. If $R$ is furthermore smooth over $k$, $char(k) \neq 2$ and $n \geq 3$, then \cite[Remark 7.10]{Mo} and \cite[Theorem 2.1]{F} imply that in fact

\begin{center}
$\mathit{Um}_{n} (R)/E_{n} (R) \cong [X, \mathbb{A}^{n}\setminus 0]_{\mathcal{H}(k)}$.
\end{center}

It was proven in \cite[Corollary 5.43]{Mo} that for a perfect field $k$ and $n \geq 3$ there is a canonical bijection
\begin{center}
$[\mathbb{A}^{n}\setminus 0, \mathbb{A}^{n}\setminus 0]_{\mathcal{H}(k)} \cong GW (k)$
\end{center}
called motivic Brouwer degree by analogy with the classical Brouwer degree in algebraic topology; here $GW(k)$ denotes the Grothendieck-Witt ring of non-degenerate symmetric bilinear forms over $k$.\\
Finally, we quickly introduce some notions of contractibility in motivic homotopy theory:
\begin{Def}
A space $\mathcal{X} \in Spc_k$ is called $\mathbb{A}^{1}$-contractible if the unique morphism $\mathcal{X} \rightarrow Spec(k)$ is an $\mathbb{A}^{1}$-weak equivalence.
\end{Def}

\begin{Def}
Let $X \in Sm_{k}$ and $x$ be a closed point. Then the pointed space $(X,x)$ is stably $\mathbb{A}^{1}$-contractible if ${\mathbb{P}_{k}^{1}}^{\wedge n} \wedge (X,x)$ is an $\mathbb{A}^{1}$-contractible space for some $n \geq 0$.
\end{Def}

Now let us work over a base field $k$ which admits an embedding $\iota: k \hookrightarrow \mathbb{C}$. By means of such an embedding, one may associate a complex manifold $X_{\iota}^{an}$ to any smooth variety $X$ over $k$.

\begin{Def}
Let $k$ be a field which admits an embedding into $\mathbb{C}$. A smooth affine $k$-variety $X$ is called topologically contractible if the manifold $X_{\iota}^{an}$ is a contractible topological space for any embedding $\iota: k \hookrightarrow \mathbb{C}$.
\end{Def}

\section{Results}\label{3}

We can now prove the main results in this paper. As usual, $R$ always denotes a commutative ring with $2 \in R^{\times}$.

%\begin{Rem}\label{R3.1}
%By Theorem \ref{Spin-criterion} the comparison map $Um_{3}(R)/SL_{3}(R) \rightarrow U_{5}(R)/Spin_{6}(R)$ is bijective if and only if $SL_{4}(R) = SL_{3}(R)E_{4}(R)Sp(\chi)$ (which is the case if and only $SL_{4}(R) = Sp(\chi)E_{4}(R)SL_{3}(R)$ by Remark \ref{RemarkPermutingGroups}). It is easy to see that this criterion is equivalent to the equality of the $SL_4 (R)$-orbits and $Sp (\chi)$-orbits of $e_4 = {(0,0,0,1)}^{t}$ for any invertible alternating matrix $\chi$ of rank $4$ and Pfaffian $1$. This is precisely the criterion for the injectivity of the Vaserstein symbol modulo SL found in \cite{Sy2} for Noetherian rings of dimension $\leq 4$. As a consequence, one has $Spin_{6}(R) = St (x,y) Epin_{4}(R)SL_{3}(R)$ if for such a ring the Vaserstein symbol modulo SL is bijective. 
%\end{Rem}

\begin{Thm}\label{bijectivity}
The map $Um_{3}(R)/SL_{3}(R) \rightarrow U_{5}(R)/Spin_{6}(R)$ is a bijection if
\begin{itemize}
\item $R$ is a Noetherian ring of dimension $\leq 2$ (with $2 \in R^{\times}$)
\item $R$ is a smooth affine algebra of dimension $3$ over a perfect field $k$ such that $6 \in k^{\times}$ and $c.d.(k) \leq 1$
\item $R$ is a smooth affine algebra of dimension $4$ over an algebraically closed field $k$ with $6 \in k^{\times}$
\end{itemize}
In particular, $SL_4 (R) =  \wedge (SL_3 (R))Epin_4 (R)Sp (\chi)$ in these cases, where $\chi$ is any non-degenerate alternating form with Pfaffian $1$.
\end{Thm}
\begin{proof}
The bijectivity of the map $Um_{3}(R)/SL_{3}(R) \rightarrow U_{5}(R)/Spin_{6}(R)$ follows from Remark \ref{Vaserstein} and the bijectivity of the Vaserstein symbol modulo SL: The case of a Noetherian ring of dimension $\leq 2$ follows easily from the criteria for the injectivity and surjectivity of the Vaserstein symbol modulo SL proven in \cite[Theorems 3.2 and 3.6]{Sy2}, the other two cases are covered by \cite[Theorems 3.9 and 3.11]{Sy3}. The second statement follows directly from Theorem \ref{Spin-criterion} and Remark \ref{Vaserstein}.
\end{proof}

There is a sufficient criterion for the equality of $Spin_{2n}(R)$-orbits and $SO_{2n}(R)$-orbits:

\begin{Prop}\label{SpinSO}
Assume $R$ is a smooth affine algebra over a field $k$ with $char(k) \neq 2$. If $R^{\times} = k^{\times}$, the $2$-torsion of $Pic(R)$ is trivial and $k$ is algebraically closed, then $Spin_{2n}(R)$-orbits and $SO_{2n}(R)$-orbits of $U_{2n-1}(R)$ coincide.
\end{Prop}
\begin{proof}
Following \cite[Chapter III, \S 3]{K} there is an exact sequence
\begin{center}
$0 \rightarrow R^{\times}/{(R^{\times})}^{2} \rightarrow Disc (R) \rightarrow {Pic}(R)\{2\} \rightarrow 0$.
\end{center}
Under the assumptions it follows that $Disc (R)$ is trivial and hence $Spin_{2n}(R) \rightarrow SO_{2n}(R)$ is surjective by \cite[Chapter IV, Theorem 6.2.6]{K}, which implies the equality of the orbits of $U_{2n-1}(R)$.
\end{proof}

\begin{Kor}
Let $X = Spec (R)$ be smooth affine complex variety of dimension $3$ or $4$. Assume that one of the following conditions is satisfied:
\begin{itemize}
\item[a)] $X$ is topologically contractible and $R^{\times} = \mathbb{C}\setminus 0$
\item[b)] $X$ is stably $\mathbb{A}^{1}$-contractible
\end{itemize}
Then there are bijections $Um_{3}(R)/SL_{3}(R) \xrightarrow{\cong} U_{5}(R)/Spin_{6}(R) \xrightarrow{\cong} U_{5}(R)/SO_{6}(R)$.
\end{Kor}

\begin{proof}
It was famously proven in \cite{G} that $Pic(R)$ is trivial whenever $X=Spec(R)$ is a topologically contractible smooth affine complex variety (of any dimension); a well-known representability result (cf. \cite[Section 6.1]{V}) for motivic cohomology groups (and, in particular, for Picard groups) in stable motivic homotopy theory directly implies that the Picard group of a stably $\mathbb{A}^1$-contractible smooth affine complex variety (of any dimension) is also trivial. Hence the statement follows from Theorem \ref{bijectivity} and Proposition \ref{SpinSO}.
\end{proof}

Recall that for any field $k$ with characteristic $\neq 2$, we denote by $S_{2n-1}$ the smooth affine $k$-algebra $k[x_{1},...,x_{n},y_{1},...,y_{n}]/\langle \sum_{i=1}^{n} x_{i}y_{i} - 1 \rangle$. Furthermore, we let $Q_{2n-1} = Spec(S_{2n-1})$.\\
If $n \geq 3$ is odd and $k$ is infinite perfect with $char(k) \neq 2$, we will see below that the orbit space $Um_{n}(S_{2n-1})/SL_{n}(S_{2n-1})$ as well as the abelian group $\ker (GW_{0}^{n-1}(S_{2n-1}) \rightarrow K_{0}(S_{2n-1}))$ can be computed explicitly in terms of the Witt group $W(k)$ of non-degenerate symmetric bilinear forms by using results proven in \cite{AF1} and \cite{AF2}; as a consequence, this enables us to understand the map $\Psi_n$ modulo SL as follows:

\begin{Thm}\label{Identification}
If $n \geq 3$ is odd, $k$ is infinite perfect with $char(k) \neq 2$, there is a canonical bijection $Um_{n}(S_{2n-1})/SL_{n}(S_{2n-1}) \cong W(k) \times_{\mathbb{Z}/2\mathbb{Z}} \mathbb{Z}/(n-1)!\mathbb{Z}$, where we consider the fiber product with respect to the rank modulo $2$ map $W(k) \rightarrow \mathbb{Z}/2\mathbb{Z}$ and the canonical surjection $\mathbb{Z}/(n-1)!\mathbb{Z} \rightarrow \mathbb{Z}/2\mathbb{Z}$, while there is a canonical isomorphism $\ker (GW_{0}^{n-1}(S_{2n-1}) \rightarrow K_{0}(S_{2n-1})) \cong W(k)$. Under these identifications the map $\Psi_{n}$ modulo SL for $R=S_{2n-1}$ then corresponds to the projection $W(k) \times_{\mathbb{Z}/2\mathbb{Z}} \mathbb{Z}/(n-1)!\mathbb{Z} \rightarrow W(k)$.
\end{Thm}

\begin{proof}
By the proof of \cite[Theorem 4.10]{AF1}, the Witt-valued Brouwer degree map
\begin{center}
$Um_{n}(S_{2n-1})/E_{n}(S_{2n-1}) = [Q_{2n-1},\mathbb{A}^{n}\setminus 0]_{\mathcal{H}_{\bullet}(k)} = GW(k) \rightarrow W(k)$
\end{center}
and the rank map
\begin{center}
$Um_{n}(S_{2n-1})/E_{n}(S_{2n-1}) = [Q_{2n-1},\mathbb{A}^{n}\setminus 0]_{\mathcal{H}_{\bullet}(k)} = GW(k) \rightarrow \mathbb{Z}$
\end{center}
induce the identification $Um_{n}(S_{2n-1})/SL_{n}(S_{2n-1}) \cong W(k) \times_{\mathbb{Z}/2\mathbb{Z}} \mathbb{Z}/(n-1)!\mathbb{Z}$.\\
By \cite[Proposition 3.4.3]{AF2}, the group $\ker (GW_{0}^{n-1}(S_{2n-1}) \rightarrow K_{0}(S_{2n-1}))$ is a free $W(k)$-module of rank $1$ generated by $\eta\Psi_{n}(x,y)$, where $x={(x_{1},...,x_{n})}^{t}$ and $y={(y_{1},...,y_{n})}^{t}$. Now consider the composite
\begin{center}
$Um_{n}(S_{2n-1})/E_{n}(S_{2n-1}) = [Q_{2n-1},\mathbb{A}^{n}\setminus 0]_{\mathcal{H}_{\bullet}(k)} = GW(k) \xrightarrow{\eta\Psi_n} \ker (GW_{0}^{n-1}(S_{2n-1}) \rightarrow K_{0}(S_{2n-1})) = W(k) \cdot \eta\Psi_{n}(x,y)$.
\end{center}
Since the Witt-valued Brouwer degree of $(x,y)$ is $1 \in W(k)$ and $\eta\Psi_{n}$ is $GW(k)$-linear, we can identify $\eta\Psi_n$ with the projection $GW(k) \rightarrow W(k)$ above and the theorem follows.
\end{proof}

%The proof of \cite[Proposition 3.4.3]{AF2} shows that $\eta$ is $GW(k)$-linear, because it identifies $W_{SL} (R)$ with $GW(k)/H(SK_{1}(R))$; the proof shows that the latter is indeed $GW(k)/\lange h \rangle = W(k)$ (this is a $GW(k)$-module as the ideal of $GW(k)$ generated by h coincides with the subgroup generated by h).

\begin{Kor}\label{S5}
If $k$ is infinite perfect with $char(k) \neq 2$, then $\Psi_{3}$ modulo SL is a bijection for $R=S_{5}$. In particular, $Um_{3}(S_{5})/SL_{3}(S_{5}) \xrightarrow{\cong} U_{5}(S_{5})/Spin_{6}(S_{5})$ is a bijection.
\end{Kor}

\begin{proof}
This follows immediately from Proposition \ref{detect} and Theorem \ref{Identification}.
\end{proof}

\begin{Thm}\label{MTPart1}
If $R=S_7$ and $k$ is infinite perfect with $char(k) \neq 2$, then $E_n (R)$ acts transivitely on $Um_{n}(R)$ for $n \geq 5$; furthermore, $SL_4 (R) e_4 \neq Sp_4 (R) e_4$, where $e_4 = {(1,0,0,0)}^{t} \in R^4$.
\end{Thm}

\begin{proof}
The first statement is clear as $Um_n (R)/E_n (R) \cong [Q_{7}, Q_{2n-1}]_{\mathcal{H}(k)} = \ast$ for $n \geq 5$ by \cite[Corollary 5.43]{Mo}.\\
For the second statement, recall that $Um_4 (R)/SL_4 (R) \cong \mathcal{V}^{o}_{3} (R) \cong \mathbb{Z}/3!\mathbb{Z}$ by \cite[Theorem 4.8]{AF1}, where the representatives for the $3!=6$ distinct orbits are exactly given by the vectors ${(x_{1}^{m},x_{2},x_{3},x_{4})}^{t}$ for $1 \leq m \leq 6$. By Suslin's famous theorem on completion of unimodular vectors, the vectors ${(x_{1}^{6},x_{2},x_{3},x_{4})}^{t}$ and ${(x_{1}^{12},x_{2},x_{3},x_{4})}^{t}$ are completable to an invertible matrix of determinant $1$.\\
In particular, the vectors ${(x_{1}^{6},x_{2},x_{3},x_{4})}^{t}$ and ${(x_{1}^{12},x_{2},x_{3},x_{4})}^{t}$ are in the $SL_4 (R)$-orbit of $e_4 = {(1,0,0,0)}^{t} \in R^4$. We show that at least one of these vectors is not in the $Sp_4 (R)$-orbit of $e_4$: Assume that both of them are in fact in the $Sp_4 (R)$-orbit of $e_4$. Then there would be a symplectic matrix transforming the unimodular vector ${(x_{1}^{6},x_{2},x_{3},x_{4})}^{t}$ into ${(x_{1}^{12},x_{2},x_{3},x_{4})}^{t}$. But by \cite[Proposition 3.5]{F} this would mean that the associated symplectic modules of rank $2$ are isomorphic, which is not true by \cite[Remark 4.15]{AF1}. This proves the second statement.
\end{proof}

\begin{Thm}\label{MTPart2}
If $R=S_7$ and $k$ is infinite perfect with $char(k) \neq 2$, then the comparison map $Um_3 (R)/SL_3 (R) \rightarrow U_{5}(R)/Spin_6 (R)$ is not injective (and not bijective) and $\Psi_3$ modulo SL induces a bijection $U_{5}(S_{7})/Spin_{6}(R) \xrightarrow{\cong} W_{SL}(R)$.
\end{Thm}

\begin{proof}
Consider the factorization
\begin{center}
$\Psi_{3}: Um_{3}(R)/SL_{3}(R) \rightarrow U_{5}(R)/Spin_{6}(R) \rightarrow \tilde{A}_{4}(R)/SL_{4}(R) \rightarrow W_{SL}(R)$.
\end{center}
By Remark \ref{Vaserstein}, the map $U_{5}(R)/Spin_{6}(R) \rightarrow \tilde{A}_{4}(R)/SL_{4}(R)$ is bijective and $\Psi_{3}$ modulo SL coincides with the Vaserstein symbol modulo SL associated to $\theta: R \xrightarrow{\cong} \det (R^2), 1 \mapsto e_{1} \wedge e_{2}$, via the canonical identification $\tilde{\mathit{V}}_{SL} (\mathit{R}) \cong \mathit{W}_{\mathit{SL}} (\mathit{R})$ from \cite[Section 2.C]{Sy2}. As $E_{n}(R)$ acts transitively on $Um_{n}(R)$ for $n \geq 5$ by Theorem \ref{MTPart1}, the Vaserstein symbol modulo SL is surjective (cf. \cite[Theorem 4.5]{Sy1}); hence the same holds for the map $\tilde{A}_{4}(R)/SL_{4}(R) \rightarrow W_{SL}(R)$. As $E_{2n}(R)$ acts transitively on $Um_{2n}(R)$ for $n \geq 3$, also the group $Sp (\chi)$ acts transitively on $Um_{2n}(R)$ for $n \geq 3$ and for any invertible alternating matrix $\chi$ of rank $2n$ (cf. \cite[Lemma 5.5]{SV}). Therefore the injectivity of the map $\tilde{A}_{4}(R)/SL_{4}(R) \rightarrow W_{SL}(R)$ follows from \cite[Lemma 3.4]{Sy2}; note that \cite[Lemma 3.4]{Sy2} holds for any commutative ring and does not use the dimension assumption at the beginning of \cite[Section 3.B]{Sy2}. So the map $\tilde{A}_{4}(R)/SL_{4}(R) \rightarrow W_{SL}(R)$ is bijective.\\
Similarly, as $E_{n}(R)$ acts transitively on $Um_{n}(R)$ for $n \geq 5$, the statement of \cite[Theorem 3.6]{Sy2} still holds for $R$ even though $\dim (R) > 4$. In particular, it follows easily that the map $\Psi_{3}$ modulo SL can only be injective if $SL_4 (R) e_4 = Sp_4 (R) e_4$, where $e_4 = {(1,0,0,0)}^{t} \in R^4$. But, by Theorem \ref{MTPart1} above, we already know that $SL_4 (R) e_4 \neq Sp_4 (R) e_4$. So $\Psi_{3}$ modulo SL is not injective and hence $Um_{3}(R)/SL_{3}(R) \rightarrow U_{5}(R)/Spin_{6}(R)$ cannot be injective either. This finishes the proof.
\end{proof}

\end{document}